\documentclass{article}
\usepackage{amsmath}
\usepackage{amsthm}
\usepackage{amssymb}
\usepackage{delarray}
\usepackage[dvips]{graphics}
\usepackage{epsfig}
\usepackage{color}
\usepackage{mathrsfs}
\newtheorem{lemma}{Lemma}
\newtheorem{thm}{Theorem}
\newtheorem{defn}{Definition}
\newtheorem{conj}{Conjecture}

\newtheorem{cor}{Corollary}

\title{Discriminators of quadratic polynomials}
\author{Soohyun Park\\ Massachusetts Institute of Technology\\ Cambridge, MA 02139 \\
\texttt{soopark@mit.edu}}

\date{06/30/13}
\begin{document}

\maketitle

\begin{abstract}
Given $f \in \mathbb{Z}[x]$ and $n \in \mathbb{Z^{+}}$, the $\emph{discriminator}$ $D_f(n)$ is the smallest positive integer $m$ such that $f(1), \ldots, f(n)$ are distinct mod $m$. In a recent paper, Z.-W. Sun proved that $D_f(n) = d^{\lceil \log_d n \rceil}$ if $f(x) = x(dx - 1)$ for $d \in \{2, 3\}$. We extend this result to $d = 2^r$ for any $r \in \mathbb{Z}^{+}$ and find that $D_f(n) = 2^{\lceil \log_2 n \rceil}$ in this case. We also provide more general statements for $d = p^r$, where $p$ is a prime. In addition, we present a potential method for generating prime numbers with discriminators of polynomials which do not always take prime values. Finally, we describe some general statements and possible topics for study about the discriminator of an arbitrary polynomial with integer coefficients.
\end{abstract}

\section{Introduction}

\begin{defn} \cite{S, Z} 
Let $f(x) \in \mathbb{Z}[x]$ and $n \in \mathbb{Z^{+}}$. The \emph{discriminator} of $f$ is $D_f(n) := \min\{m \in \mathbb{Z^{+}}: f(1), f(2), \ldots, f(n) \text{ are distinct modulo $m$}\}$. \\ If no such $m$ exists, we set $D_f(n) = \infty$.
\end{defn}

The discriminator was first defined for $f(x) = x^2$ as the smallest positive integer $m$ such that $1^2, 2^2, \ldots, n^2$ are pairwise distinct modulo $m$. It was originally involved in determining an efficient algorithm for computing the square roots of a long sequence of integers for a problem in computer simulation (see $\cite{ABM}$ for more information). Other polynomials for which $D_f(n)$ has been studied include powers of $x$ and Dickson polynomials of a degree relatively prime to 6 (see $\cite{Z}$). Most of the values of $D_f(n)$ are quite complicated. However, there are some cases where $D_f(n)$ has relatively simple values. For example, $D_f(n) $ is the smallest integer $m \ge 2n$ such that $m = p$ or $m = 2p$ for some odd prime $p$ if $f(x) = x^2$ ($n > 4$) $\cite{ABM}$. In addition, in a recent paper by Z.-W. Sun $\cite{S}$, it was found that for some quadratic polynomials $f$, $D_f(n)$ is a prime that has a simple description. For example, $D_f(n)$ is the least prime greater than $2n - 2$ if $f(x) = 2x(x - 1)$. Note that this can theoretically be used to generate all primes, but not feasibly. \\

Most of the discriminators of the polynomials considered in $\cite{S}$ either take prime values or are of the form $d^{\lceil \log_d n \rceil}$, where $d \in \mathbb{Z^{+}}$. For example, Sun proves that if $f(x) = x(dx - 1)$ for $d = 2$ or $d = 3$, then $D_f(n) = d^{\lceil \log_d n \rceil}$. We generalize the $d = 2$ case in Section 2 and prove that the discriminator is equal to $2^{\lceil \log_2 n \rceil}$ if $d = 2^r$ for any $r \in \mathbb{Z^{+}}$. We also make some general statements about the case where $d = p^r$ in Section 3, where $p$ is a prime and $r \in \mathbb{Z^{+}}$. In this case, we provide a potential method for finding a function which only takes prime values. Finally, we suggest potential directions for future study in Section 4 and consider discriminators of arbitrary polynomials with integer coefficients after they are multiplied by a constant. This allows us to obtain estimates for the sizes of the prime values which the discriminators considered in $\cite{S}$ take. \\

\section{A result on discriminators of quadratic polynomials}

In this section, we shall prove that $D_f(n) = d^{\lceil \log_2 n \rceil}$ if $f(x) = x(2^r x - 1)$, where $r \in \mathbb{Z}^{+}$.  We first recall a result of Z.-W. Sun \cite{S}.

\begin{thm}[Sun]
Let $d \in \{2, 3\}$ and $n \in \mathbb{Z^{+}}$. If $f(x) = x(dx - 1)$, then $D_f(n) = d^{\lceil \log_d n \rceil}$.
\end{thm}

We extend this theorem to the case $d = 2^r$ for any $r \in \mathbb{Z^{+}}$.

\begin{thm}
Let $d = 2^r$ with $r \in \mathbb{Z^{+}}$ and $f(x) = x(dx - 1)$. Then, $D_f(n) = 2^{\lceil \log_2 n \rceil}$. 
\end{thm}

Before we prove this theorem, we will give an upper bound for $D_f(n)$ in the more general case where $d = p^r$ for some prime $p$.

\begin{lemma}
Let $d = p^r$ ($p$ prime) with $r \in \mathbb{Z^{+}}$ and $f(x) = x(dx - 1)$. Then $D_f(n) \le p^{\lceil \log_p n \rceil}$. 
\end{lemma}

\begin{proof}
It is sufficient to show that $f(1), \ldots, f(n)$ are distinct modulo $p^{\lceil \log_p n \rceil}$. \\

Suppose there exist $k$ and $l$ with $1 \le k < l \le n$ such that $f(k) \equiv f(l)$ (mod $p^{\lceil \log_p n \rceil}$). Note that $f(l) - f(k) = (l - k)(d(l + k) - 1)$. This means that $(l - k)(d(l + k) - 1) \equiv 0$ (mod $p^{\lceil \log_p n \rceil}$). Since $d(l + k) - 1$ is not divisible by $p$ and $p | d$, we have $p^{\lceil \log_p n \rceil} | l - k \Rightarrow p^{\lceil \log_p n \rceil} \le l - k < n$. Then, $\log_p n > \lceil \log_p n \rceil$, which is a contradiction. Therefore, $l$ and $k$ cannot both be in $\{1, \ldots, n \}$ as desired.  
\end{proof}

Here is a consequence of this lemma.

\begin{cor}
Let $f(x) = x(dx - 1)$ with $d = p^r$ for some prime $p$. Then $D_f(n) = n$ when $n$ is a power of $p$.
\end{cor}

\begin{proof}
First, $D_f(n) \ge n$ for any function $f$. Let $n = p^k$. From Lemma 1, we have $D_f(n) \le p^{\lceil \log_p n \rceil} = p^k = n \Rightarrow D_f(n) = n$.
\end{proof}

Now we turn to the proof of Theorem 2.

\begin{proof} 
By Lemma 1 we have $D_f(n) \le 2^{\lceil \log_2 n \rceil}$ ($p = 2$ case). \\

We will show that if $m < 2^{\lceil \log_2 n \rceil}$, then there exist $k, l \in \{1, \ldots n \}$ distinct such that $f(l) \equiv f(k)$ (mod $m$). Note that $f(l) - f(k) = (l - k)(d(l + k) - 1)$. \\

$\bf{Case \text{ } 1:}$ $m = 2^t, t \le \lceil \log_2 n \rceil - 1$. \\

We can take $k = 1$ and $l = 2^t + 1 \le n$ and find that $m = 2^t | (l - k)(d(l + k) - 1)$ since $l - k = 2^t$. \\ 

$\bf{Case \text{ } 2:}$ $m$ is odd. \\

To verify that it is possible to find $k, l$ distinct such that $(l - k)(d(l + k) - 1) \equiv 0 \text{ (mod $m$)}$ in this case, observe that

\begin{align*}
d(l + k) - 1 \equiv 0 \text{ (mod $m$)} \\
\Leftrightarrow d(l + k) \equiv 1 \text{ (mod $m$)} \\
\Leftrightarrow l + k \equiv \bar{d} \text{ (mod $m$),}
\end{align*}

\noindent
where $\bar{d}$ is the least positive remainder of the inverse of $d$ mod $m$ (which exists because gcd($d$, $m$) = 1). Note that $m < 2^{\lceil \log_2 n \rceil} < 2n$. Since $\bar{d} < m < 2n$, there are $k, l$ such that $l + k = \bar{d}$ if $\bar{d} \ge 3$. If $\bar{d} = 1$, we can find $k, l$ such that $k + l = m + \bar{d}$ since $m < 2^{\lceil \log_2 n \rceil} < 2n \Rightarrow 2n - m \ge 2 \Rightarrow m + 1 < 2n$. The same can be done for $\bar{d} = 2$ if $2n - m > 2 \Rightarrow m + 2 < 2n$. We have $2n - m = 2$ if and only if $m + 1 = 2^{\lceil \log_2 n \rceil}$ and $2^{\lceil \log_2 n \rceil} + 1 = 2n$. However, the second statement is impossible since $2^{\lceil \log_2 n \rceil} + 1$ is odd (for $n > 1$) and $2n$ is even. \\

$\bf{Case \text{ } 3:}$ $m = 2^a q$ ($a \ge 1$, $q \ge 3$ odd). \\ 

We have $(l - k)(d(l + k) - 1) \equiv 0 \text{ (mod $2^a q$)}$ if and only if $(l - k)(d(l + k) - 1) \equiv 0 \text{ (mod $2^a$)}$ and $(l - k)(d(l + k) - 1) \equiv 0 \text{ (mod $q$)}$. We claim that there exist $k, l$ such that $l + k \equiv \bar{d} \text{ (mod $q$)}$ and $l - k \equiv 0 \text{ (mod $2^a$)}$ such that $2^a | l - k$ and $q | d(l + k) - 1$ $\Rightarrow  m = 2^a q | (l - k)(d(l + k) - 1)$. Here, $\bar{d}$ is the least positive value of the inverse of $d$ mod $q$. This reduces to 

\begin{align*}
l + k = \bar{d} + vq \\
l - k = w \cdot 2^a.
\end{align*}

Solving for $l$ and $k$, we obtain $l = \frac{1}{2}(\bar{d} + vq) + w \cdot 2^{a - 1}$ and $k = \frac{1}{2}(\bar{d} + vq) - w \cdot 2^{a - 1}$. Since $k < l$ and $w > 0$, we have $w \ge 1$. Also, $\bar{d}$ and $v$ have the same parity since $k, l \in \mathbb{Z}$ if and only if $\bar{d} + vq$ is even (given $w \in \mathbb{Z}$) and $q$ is odd. Say that we choose $w = 1$ and $v$ to be the smallest integer such that $qv + \bar{d} > 2^a$. Then, we have that $qv + \bar{d} \le 2^a + q - 1$. However, $qv + \bar{d}$ may not be even in this case and we may need to add another copy of $q$ to change parity. This implies that $qv + \bar{d} \le 2^a + 2q - 1$ after we add the condition that $qv + \bar{d}$ is even. Since $2^a + 2q - 1$ is odd, we have an upper bound of $2^a + 2q - 2$. If $v$ is negative, then $k$ or $l$ may be negative. In this case, we choose $v = 0$ and get $qv + \bar{d} = \bar{d}$. Taking $w = 1$ and $v$ to be the smallest $\emph{nonnegative}$ integer such that $qv + \bar{d} > 2^a$ and $qv + \bar{d}$ is even, we have $qv + \bar{d} \le \max(2^a + 2q - 2, \bar{d}) \le \max(2^a + 2q - 2, q - 1) = 2^a + 2q - 2$. This means that $2l = qv + \bar{d} + 2^a \le 2^a + 2^a + 2q - 2 = 2^{a + 1} + 2q - 2$. Since $m = 2^a q < 2^{\lceil \log_2 n \rceil}$, we have $2^{a + 1} + 2q - 2 \le 2n$, which implies that $2l \le 2^{a + 1} + 2q - 2 \le 2n \Rightarrow l \le n$. \\

This means that integer solutions exist for $v$ and $w$ and suitable values of $k$ and $l$ exist such that $(l - k)(d(l + k) - 1) \equiv 0 \text{ (mod $m$)}$ when $m = 2^a q$ with $a \ge 1$ and $q$ odd ($q \ge 3$). Therefore, $D_f(n) \ge 2^{\lceil \log_2 n \rceil} \Rightarrow D_f(n) = 2^{\lceil \log_2 n \rceil}$.

\end{proof}

\section{Properties of the discriminator in the $d = p^r$ case and a potential method to generate primes}

In the previous section, we used a result about the general $d = p^r$ case in order to prove Theorem 2. Lemma 1 stated that $D_f(n) \le p^{\lceil \log_p n \rceil}$ and it followed from this lemma that $D_f(n) = n$ when $n$ is a power of $p$. We can make some more specific observations about this case after finding the value of $D_f(n)$ for various values of $d$ and $n$ using a computer program. First, $D_f(n)$ behaves similarly to $p^{\lceil \log_p n \rceil}$ when $d$ is a power of a small prime $p$. Moreover, there is still a significant clustering around powers of $p$ even for relatively large $p$. In addition, the $D_f(n)$ where $f(x) = x(dx - 1)$ for $d = p, p^2, p^3, \ldots$ ($p$ prime) appear to behave very similarly to each other. Generally, the value of $D_f(n)$ seems to deviate more from $p^{\lceil \log_p n \rceil}$ when $p$ is large or a large power of $p$ is used. \\





We used a computer program to determine $D_f(n)$ for different values of $d = p^r$ and $n$ by looping through a bound on $D_f(n)$ in terms of $n$ and checking whether $f(1), \ldots, f(n)$ were distinct modulo $m$ for each $m$ in this interval. We will now give the interval used and describe how it was obtained. Given $d$, we want to show that $n \le D_f(n) < dn$ for $f(x) = x(dx - 1)$. If $m < n$, there exist $f(k)$ and $f(l)$ with $1 \le k < l \le n$ such that $f(k) \equiv f(l) \text{ (mod $m$)}$ by the pigeonhole principle. So, $D_f(n) \ge n$. The upper bound follows from Lemma 1 since $D_f(n) \le p^{\lceil \log_p n \rceil} < pn \le dn$. Some of the values tested are recorded in the tables below. \\ 


\begin{table}[ht]
\caption{Discriminator values for $f(x) = x(3^3 x - 1)$, $n = 1, \ldots, 300$. Note that 223, 541, 659, and 709 are prime.} 
\centering
\begin{tabular}{| c | c | c | c |}
\hline
$n$ & $D_f(n)$ & $n$ & $D_f(n)$ \\
\hline
1 & 1 & 82 - 97 & 223 \\
2 - 3 & 3 &  98 - 243 & 243\\
4 - 9 & 9 & 244 - 260 & 541 \\
10 - 27 & 27 & 261 - 270 & 659 \\
28 - 81 & 81 & 271 - 300 & 709\\
\hline
\end{tabular} 
\end{table}

\begin{table}[ht]
\caption{Discriminator values for $f(x) = x(7^2 x - 1)$, $n = 1, \ldots, 300$. Note that 37, 41, 131, 157, 197, 229, and 331 are prime.} 
\centering
\begin{tabular}{| c | c | c | c |}
\hline
$n$ & $D_f(n)$ & $n$ & $D_f(n)$\\
\hline
1 & 1 & 19 - 49 & 49\\
2 & 3 & 50 - 61 & 131 \\
3 - 7 & 7 & 62 - 70 & 157 \\
8 & 16 & 71 - 96 & 197 \\
9 & 21 & 97 - 107 & 229 \\
10 - 17 & 37 & 108 - 152 & 331 \\
18 & 41 & 153 - 300 & 343\\
\hline
\end{tabular} 
\end{table}

\pagebreak

\begin{table}[ht]
\caption{Discriminator values for $f(x) = x(29x - 1)$, $n = 1, \ldots, 500$. Note that all values except 1, 15, and 841 are primes.} 
\centering
\begin{tabular}{| c | c | c | c | c | c |}
\hline
$n$ & $D_f(n)$ & $n$ & $D_f(n)$ & $n$ & $D_f(n)$ \\
\hline
1 & 1 & 48 - 61 & 131 & 197 & 457 \\
2 & 3 & 62 & 151 & 198 - 223 & 479 \\
3 - 4 & 7 & 63 - 72 & 167 & 224 - 225 & 503 \\
5 & 15 & 73 - 75 & 199 & 226 - 252 & 523 \\
6 - 10 & 19 & 76 - 112 & 233 & 253 - 277 & 619 \\
11 - 29 & 29 & 113 - 121 & 271 & 278 - 304 & 653 \\
30 - 34 & 73 & 122 & 283 & 305 - 358 & 769 \\
35 - 43 & 97 & 123 - 168 & 349 & 359 - 385 & 827 \\
44 - 47 & 109 & 169 - 196 & 421 & 386 - 500 & 841 \\
\hline
\end{tabular} 
\end{table}

Based on the values tested, we raise the following conjecture about the behavior of $D_f(n)$.

\begin{conj}
For $f(x) = x(dx - 1)$ and $d = p^r$, $D_f(n)$ is either a prime number or $p^{\lceil \log_p n \rceil}$ for sufficiently large $n$.
\end{conj}

Note that the conditions on the discriminator given in this conjecture are similar to conditions on $m$ given in $\cite{S}$ in order to have $f(1), \ldots, f(n)$ distinct modulo $m$ if $f(x) = x(x - 1)$. 

\begin{thm} [Sun]
Let $f(x) = x(x - 1)$. If $m$ and $n$ are integers such that $f(1), \ldots, f(n)$ are distinct modulo $m$, then $m$ is a prime or a power of two if $n \ge 15$ and $m \le 2.4 n$.
\end{thm} 

When $D_f(n)$ did not take values which were powers of $p$, almost all of the values taken were prime numbers. If a condition can be found for when these values occur, this may lead to additional methods to generate primes using functions whose discriminators which do not always take prime values. This may give relatively simple functions beyond the discriminators considered in $\cite{S}$ which take prime values. \\

\section{General statements about $D_f(n)$ for $f \in \mathbb{Z}[x]$ and future directions}

While many different patterns were observed above, there is still no general explanation for them and why the discriminator takes prime values in certain cases. In other words, it remains to be shown whether this has anything to do with the polynomials chosen or the discriminator itself. So, it may also be useful to try to determine how $D_f(n)$ changes when an operation is performed on $f$. This could be used to relate discriminators of different functions to each other in order to find some general structure for discriminators of polynomials such as expressing $D_{f \circ g}(n)$ or $D_{fg}(n)$ in terms of $D_f(n)$ and $D_g(n)$. This could enable us to deduce certain properties of the discriminators of some polynomials without directly computing them. \\

One such operation is multiplying $f$ by a constant. In this instance, it is sufficient to look at the case where the constant is prime since we can compose multiplication by other constants by multiplication of primes. If $D_f(n)$ is not divisible by $p$, then $D_{pf}(n)$ is the same as $D_f(n)$. If $D_f(n)$ is divisible by a  prime $p$, then $D_{pf}(n)$ can take quite a different form. For example, the discriminator of $\frac{x(x - 1)}{2}$ is always a power of 2, whereas the discriminator of $x(x - 1)$ can take arbitrarily large prime values in addition to powers of 2. Another example relates the discriminator of $4x(4x - 1)$ and the discriminator of $x(4x - 1)$ to each other. Whereas the discriminator of $4x(4x - 1)$ only takes prime values (see $\cite{S}$), the discriminator of $x(4x - 1)$ is always a power of 2 (see Theorem 2). \\

\begin{thm}
Let $p$ be a prime. Then, $D_f(n) \le D_{pf}(n) \le p D_f(n)$. 
\end{thm}

\begin{proof}
Take $k, l$ such that $1 \le k < l \le n$. If $f(l) - f(k) \equiv 0 \text{ (mod $m$)}$, then $p(f(l) - f(k)) \equiv pf(l) - pf(k) \equiv 0 \text{ (mod $m$)}$. So, $D_f(n) \le D_{pf}(n)$. Also, $pf(l) - pf(k) \equiv p(f(l) - f(k)) \equiv 0 \text{ (mod $p D_f(n)$)}$ if and only if $f(l) - f(k) \equiv 0 \text{ (mod $D_f(n)$)}$, which is impossible. This means that $D_{pf}(n) \le p D_f(n)$.
\end{proof}

In $\cite{S}$, the discriminator for $4x(4x - 1)$ was found to be the least prime $p > \frac{8n - 4}{3}$ with $p \equiv 1 \text{ (mod 4)}$. From Theorem 2, we have that the discriminator for $x(4x - 1)$ is $2^{\lceil \log_2 n \rceil}$. Using the bound in Theorem 4, we find that $\frac{8n - 4}{3} < p < 4 \cdot 2^{\lceil \log_2 n \rceil} < 4 \cdot 2n = 8n$. Similar steps can be taken for discriminators of $18x(3x - 1)$ and $x(3x - 1)$ to find that $3n < p < 54 n$, where $p$ is the least prime greater than $3n$ congruent to 1 mod 3 $\cite{S}$. So, we can obtain some estimates for the sizes of the prime values which some of the discriminators considered in $\cite{S}$ take. However, we do not have more precise inequalities or equalities relating the quantities to each other. We could try to find more specific patterns by considering the values of the discriminator obtained by multiplying the polynomials considered in $\cite{S}$ by various primes. Another possible place to start is with polynomials of the form $f(x) = x^j$, where $D_f(n)$ has a relatively simple structure. In particular, Bremser, Schumer, and Washington $\cite{BSW}$ proved the following result. \\

\begin{thm} $\cite{BSW}$
Let $f(x) = x^j$. If $j$ is odd, $D_f(n) = \min \{k : k \ge n, k \text{ squarefree}, \\ \text{ gcd}(\phi(k), j) = 1 \}$. If $j$ is even, $D_f(n) = \min \{k : k \ge 2n, k = q \text{ or } 2q, \text{ } q \text{ prime}, \\ \text{ gcd}(\phi(k), j) = 2 \}$.
\end{thm}

In the case where $j$ is odd, $D_f(n)$ is the same for any $j$ that have the same prime factors. Let $f(x) = x^r$ and $g(x) = x^s$, where $r, s \in \mathbb{Z^{+}}$ are odd. Then $f \circ g = g \circ f = x^{rs} \Rightarrow D_f(n) = D_g(n) = D_{f \circ g}(n) = D_{g \circ f}(n)$ if $r$ and $s$ have the same prime factors. \\

Some possible directions for future research include looking at function composition in more detail or examining other operations on functions. In the case of multiplying $f(x) = x^j$ by a constant, one approach is to modify the conditions on $m$ in order to have $f(1), \ldots, f(n)$ distinct mod $m$, which are given in $\cite{C}$.

\section{Acknowledgements}
This research was conducted at the University of Minnesota Duluth REU program, supported by NSF/DMS grant 1062709 and NSA grant H98230-11-1-0224. I would like to thank Joe Gallian for his encouragement and creating such a great environment for research at UMD. I would also like to thank Krishanu Sankar and Sam Elder for their help with my research. I would especially like to thank Ben Bond, Krishanu Sankar, David Moulton, and Tim Chow for very helpful discussions at various points of this project.


\begin{thebibliography}{5}
\bibitem{ABM} L. K. Arnold, S. J. Benkoski, B. J. McCabe, The discriminator (a simple application of Bertrand's postulate), Amer. Math. Monthly 92 (1985) 275 - 277.
\bibitem{BSW} P. S. Bremser, P. D. Schumer, L. C. Washington, A note on the incongruence of consecutive integers to a fixed power, J. Number Theory 35 (1990) 105 - 108.
\bibitem{C} C. M. Cordes, Permutations mod $m$ in the form $x^n$, Amer. Math. Monthly 83 (1976) 32 - 33.
\bibitem{S} Z.-W. Sun, On functions taking only prime values, J. Number Theory 133 (2013) 2794 - 2812.
\bibitem{Z} M. Zieve, A note on the discriminator, J. Number Theory 73 (1998) 122 - 138.
\end{thebibliography}
\end{document}